\documentclass{amsart}
\usepackage{amssymb}
\usepackage{a4wide}
\usepackage[all]{xypic}
\usepackage{graphicx}

\newcommand{\Z}{{\mathbb Z}}
\newcommand{\N}{{\mathbb N}}
\newcommand{\Q}{{\mathbb Q}}
\newcommand{\R}{{\mathbb R}}
\newcommand{\C}{{\mathbb C}}

\newtheorem{theorem}{Theorem}[section]

\newtheorem{lemma}[theorem]{Lemma}
\newtheorem{definition}[theorem]{Definition}
\newtheorem{corollary}[theorem]{Corollary}
\newtheorem*{remark}{Remark}
\newtheorem*{thm}{Theorem}
\newtheorem*{thmA}{Theorem A}
\newtheorem*{thmB}{Theorem B}
\begin{document}

\title{The $R_\infty$ property for nilpotent quotients of surface groups.}
\author[K.\ Dekimpe]{Karel Dekimpe}
\author[D.\ Gon\c{c}alves]{Daciberg Gon\c{c}alves}
\address{KULeuven Kulak\\
E. Sabbelaan 53\\
8500 Kortrijk\\
Belgium}
\email{karel.dekimpe@kuleuven-kulak.be}
\thanks{The first author expresses his gratitude towards the department of mathematics of the University of S\~{a}o Paulo for its hospitality and  to FAPESP-Funda\c c\~ao de Amparo a Pesquisa do
Estado de S\~ao Paulo, Projeto Tem\'atico Topologia Alg\'ebrica, Geom\'etrica
e Differencial 2012/24454-8 for support during his visit to IME-USP}
\address{Departemento de Matem\'atica--IME--USP\\ 
Universidade de S\~{a}o Paulo\\ 
S\~{a}o Paulo\\
Brasil}
\email{dlgoncal@ime.usp.br}
\keywords{Reidemeister classes; $R_{\infty}$ property; nilpotent groups; surfaces; central series; Lie Algebra; eigenvalue.}
\subjclass[2010]{ Primary: 20E36; secondary: 20F16,  20F18, 20F40, 55M20.}
\maketitle

\begin{abstract}It is well known that when $G$ is the fundamental group of a closed surface of negative Euler characteristic, 
it has the $R_{\infty}$ property. In this work we  compute the least integer $c$,  {\it called the  $R_{\infty}$-nilpotency degree of   $G$}, 
  such that the group $G/ \gamma_{c+1}(G)$ has the  $R_{\infty}$ property, where $\gamma_r(G)$ is the $r$-th term of the lower central series of $G$.
  We show that   $c=4$ for $G$  the fundamental group of any  orientable closed surface $S_g$ of genus $g>1$. For the fundamental group of the non-orientable surface $N_g$ (the connected sum of $g$ projective planes)  this number is $2(g-1)$ (when $g>2$). A similar concept is introduced using the derived  series $G^{(r)}$ of a group $G$.  Namely {\it  the  $R_{\infty}$-solvability degree of  $G$}, which is
  the least integer $c$ 
  such that the group $G/G^{(c)}$ has the  $R_{\infty}$ property. We show that the fundamental group 
  of an orientable closed surface $S_g$  has $R_{\infty}$-solvability degree $2$. 
  \end{abstract}

\section{Introduction} 
Any endomorphism $\varphi$ of a group $G$ determines an equivalence relation on $G$ by setting $x\sim y \Leftrightarrow 
\exists z \in G:\; x = z y \varphi(z)^{-1}$. The equivalence classes of this relation are called Reidemeister classes or twisted conjugacy classes and their number is denoted by $R(\varphi)$. This relation has appeared in the form as defined above in 1936 in the work of K. Reidemeister
\cite{Re}. Independently, and in another context, namely the  study of  fixed points of homeomorphism of surfaces, in 1927 a similar relation, but not 
exactly the same, was introduced by J. Nielsen in \cite{N} and the equivalences classes were called isogredience classes.  A more modern 
reference which shows the connection of Reidemeister classes with fixed point theory can be found in the book of B. Jiang \cite{Ji}.  
Also observe that in case $\varphi$ is the identity automorphism, the Reidemeister classes  are exactly the 
conjugacy classes of the group $G$.  
%PROBABLY WE SHOULD SAY WHERE THESE TWISTED CONJUGACY CLASSES APPEAR IN MATHEMATICS. CAN YOU SAY SOMETHING?
A group $G$ is said to have the $R_{\infty}$ property if for every automorphism $\varphi: G \to G$  the number 
 $R(\varphi)$ is  infinite. A central problem is to decide which groups have 
 the $R_{\infty}$ property.  The study of this problem
   %the $R_{\infty}$ property of a group
  has  been a quite active  research topic in recent years. Several families of groups have been studied by many authors. 
  To see a recent list which contains  the families of groups which have been studied, with possibly a few recent exceptions,
  see the introduction and the references in \cite{FT}. Nevertheless we make explicit here the references for 
  families of groups which have been studied and are closely related with the present work.   For finitely generated abelian groups, it  is part of the folklore that 
  these groups do not have the $R_{\infty}$ property. For not finitely generated abelian groups the result is not true, see \cite{DG1} for 
  details and further results. For nonabelian groups     
% WE SHOULD  ADD A FEW (OR A LOT OF) REFERENCES HERE? CAN YOU DO THIS?
%Let   Gn, the group of pure symmetric automor-
%phisms of the free group Fn of rank n  2. Fix
%   A central problem is to decide which groups have the $R_{\infty}$  property.  
%In particular, 
we would like to mention the following result, from  \cite{LL} and \cite{Fe}: 
  
  \begin{thm}
  If $G$ is a 
  non-elementary Gromov hyperbolic group, then $G$ has the $R_{\infty}$ property.
  \end{thm}
  
  The family of the  non-elementary Gromov 
  hyperbolic group  includes the family of  free groups of finite rank greater than 1 and also the surface groups of a hyperbolic closed surface.
  For a free group of infinite rank it is shown in \cite{DG} that it  does not have the  $R_{\infty}$ property.  The family of the elementary hyperbolic groups has been studied in \cite{GW1}. Also the groups    $G_n$, the group of pure symmetric automorphisms of the free 
  group $F_n$ of rank $n\geq 2$, which have a presentation which resembles the ones for nilpotent groups, have the $R_{\infty}$ 
  property (see \cite{GK} and \cite{GS}).  
   
   It turns out that in fact some quotients 
  of $F_r$ ($1<r<\infty$) already have the $R_{\infty}$ property. More precisely the quotients of $F_r$ by the  
   $j$-$th$ term of 
  either  the lower central series or the derived central series eventually have the $R_{\infty}$ property as $j$ goes to infinity.
   To see more concrete results about  nilpotent groups,  free nilpotent groups and free solvable groups  
   with the $R_{\infty}$ property, see \cite {GW} and \cite{Ro}. Some  results from the above references  were extended in \cite{DG}. More 
   precisely from \cite[Theorems 2.5 and 3.1]{DG} we obtain:  
   
 \begin{thmA}
 Let $r, c$ be positive integers, with $r>1$. Let $G = N_{c,r}=F_r/\gamma_{c+1}(F_r)$.
Then G has property $R_{\infty}$ if and only if $ c \geq 2r$.
\end{thmA}

\begin{thmB}
Let $r > 1$ be an integer. The group $S_{r,1} = F_r/[F_r, F_r] = F_r/F_r^{(1)}=
 \Z^r$ does not have the $R_{\infty}$ property. On the other hand the groups $S_{r,k} = F_r/F_r^{(k)}$ have 
 property $R_{\infty}$ for all $k \geq 2$.
\end{thmB}

\noindent Here $\gamma_s(G)$ denotes the $s$-th term of the lower central series of $G$ and $G^{(s)}$
the $s$-th  term of the derived series of $G$.

Motivated by the two above results from  \cite{DG} we will define two numbers associated to 
a group  which have  the $R_{\infty}$ property.

\begin{definition} The $R_{\infty}$-nilpotency 
%$R_{\infty}-nilpotent$  $R_{\infty}$$-$nilpotent  $R_{\infty}$-$nilpotent$ 
degree of a group  $G$  which has the $R_{\infty}$ property is the least integer $c$ such that $G/\gamma_{c+1}(G)$
has the $R_{\infty}$  property. In case this integer does not exist then we say that $G$ 
has $R_{\infty}$-nilpotency degree infinite.
\end{definition}

According to this definition, it follows from Theorem A  that the   $R_{\infty}$-nilpotency degree of 
the free group $F_r$ of rank $r$  for $r>1$ equals $2r$.

%Observe that if if $G$ does not have the $R_{\infty}$ property, then    
% $R_{\infty}-$nilpotency degree of   $G$ is the least integer $c$ such tha

 In a similar fashion we define the  $R_{\infty}$-solvablity  degree of a group  $G$   which has the $R_{\infty}$ property.
  
 \begin{definition} 
The $R_{\infty}$-solvability degree of a group  $G$  which has the $R_{\infty}$ property is the least integer $c$ such 
that $G/G^{(c)}$ has the $R_{\infty}$  property. In case this integer does not exist then we say that $G$ 
has $R_{\infty}$-solvability degree infinite.
\end{definition}

It  follows from Theorem B that the   $R_{\infty}$-solvability  degree of 
the free group $F_r$ of rank $r$  for $r>1$ equals  $2$.

\medskip
\begin{remark}
Let $G$ be a group which has $R_\infty$-nilpotency degree $c$. Then for any $k<c$ we have that 
$G/\gamma_{k+1}(G)$ does not have property $R_\infty$, while for $k\geq c$ the groups $G/\gamma_{k+1}(G)$ 
do have property $R_\infty$. So the $R_\infty$-nilpotency degree really determines the boundary of the nilpotency class 
at which one can detect the $R_\infty$ property of a given group $G$.\\
An analogous remark holds for the $R_\infty$-solvability degree.
\end{remark}

\medskip

 Given a group $G$ which has the $R_{\infty}$ property,  besides to be a natural question to study the $R_{\infty}$-nilpotency and 
 $R_{\infty}$-solvability degree  of the group $G$, it also provides a source of new examples of groups with the $R_{\infty}$ property. 
 The present work studies this question for the case where the group $G$ is the fundamental   group of a closed surface.
So we do not consider the case where the surface is either $S^2$, $\R P^2$,  $T={\rm Torus}$ since the $R_{\infty}$ property does not hold 
for the  fundamental group of these surfaces. The case where $S=K$, the Klein bottle, 
 is easy to solve. Recall that $G=\pi_1(K)=\Z\rtimes \Z$ and that $G$ has property $R_\infty$ (see Lemma 2.1 and Theorem 2.2 in \cite{GW}).
Then we have $\gamma_s(G)=2^{s-1}\Z$ and $G^{(s)}$ is $2\Z$ if $s=1$ and the trivial subgroup if $s>1$. It then follows that $G/G^{(1)}$ 
does not have the $R_{\infty}$ property, but $G=G/G^{(s)}$ does have property $R_\infty$ for $s\geq 2$, so the $R_{\infty}$-solvability degree of $G$ is 2. 
Also from \cite{GW1} it follows that none of  the quotients $G/\gamma_{s+1}(G)=\Z_{2^s}\rtimes \Z$ have the $R_{\infty}$ property.
So, the $R_\infty$-nilpotency degree of $G$ is infinite.

\medskip

So, for the rest of this paper we will focus on the case  where  $S$ is a closed surface of negative Euler characteristic. 

\medskip 

The main results of  the manuscript are:

\medskip

{\bf Theorem} {\bf \ref{orient}.} 
{\it Let $g\geq 2$ and $\pi_g$ be the fundamental group of the orientable surface of genus $g$. \\
For any $c\geq 4$, we have that the $c$-step nilpotent quotient 
$\pi_g / \gamma_{c+1} (\pi_g)$ has property $R_\infty$. On the other hand, for $c=1,2$ and $3$,
this quotient does not have property $R_\infty$.\\
I.e.\ the $R_\infty$-nilpotency degree of $\pi_g$ is 4.}

\medskip

{\bf Theorem} {\bf\ref{nonorient}.}
{\it Let $g\geq 2$ and let $\Gamma_{g+1}$ be the fundamental group of the non-orientable surface of genus $g+1$. 
Then, for any $c\geq 2g$, we have that the $c$-step nilpotent quotient 
$\Gamma_{g+1} / \gamma_{c+1} (\Gamma_{g+1})$ has property $R_\infty$. On the other hand, for $c<2g$,
 $\Gamma_{g+1} / \gamma_{c+1} (\Gamma_{g+1})$  does not have property $R_\infty$.\\
I.e.\ the $R_\infty$-nilpotency degree of $\Gamma_{g+1}$ is $2g$. }

 \medskip

Let $G=\pi_g/\pi_g^{(2)}$, i.e. the quotient of $\pi_g$ by the $2-$th term of the  derived series of $\pi_g$. 

\medskip

{\bf Corollary
\ref{degsolv}.} Let $g\geq 2$ and let  $G=\pi_g/\pi_g^{(2)}$.\\
The group $G/\gamma_{c+1}(G)$ has the $R_{\infty}$ property, if and only if $c\geq 4$. Further,
the $R_{\infty}$-solvability degree of $\pi_g$ is 2, i.e.  
$\pi_g/\pi_g^{(c)}$ has the $R_{\infty}$ property, if and only if $c\geq 2$.

\medskip

  This work is divided into 3 sections besides the introduction. In section 2 we review and present some results about nilpotent groups 
  which we are going to use. In section 3 we consider the case  of orientable surfaces. In section 4 we consider the non-orientable case.

\section{Some preliminaries  on the $R_\infty$ property for nilpotent groups}
Let $A$ be a finitely generated abelian group with torsion subgroup $T$ and free 
abelian quotient $A/T$. Let $\varphi$ be any endomorphism of $A$, then 
$\varphi$ induces an endomorphism $\bar{\varphi}$ of $A/T$.
It is easy to see that $R(\varphi)=\infty \Leftrightarrow R(\bar\varphi) =\infty$. 
As $A/T$ is free abelian, the morphism $\bar\varphi$ can be described as a matrix (with integer entries) and so 
it makes sense to talk about the eigenvalues of $\bar\varphi$. In general we will talk about the eigenvalues of 
$\varphi$ by which we will mean the eigenvalues of $\bar\varphi$ and by the determinant of $\varphi$ we will mean 
the determinant of $\bar\varphi$.

It follows that 
$R(\varphi)=\infty$ if and only if 1 is an eigenvalue of $\bar\varphi$. 

\medskip

More generally, we consider a finitely generated nilpotent group $N$ and an endomorphism 
$\varphi$ of $N$. This morphism induces a sequence of endomorphisms
\[ \varphi_i: \gamma_i(N)/\gamma_{i+1}(N) \rightarrow \gamma_i(N)/\gamma_{i+1}(N): x \gamma_{i+1}(N) 
\mapsto \varphi(x)  \gamma_{i+1}(N).\]
As all of these groups $\gamma_i(N)/\gamma_{i+1}(N)$ are finitely generated abelian groups, we can talk about the 
eigenvalues of each of the $\varphi_i$. 
When we talk about the eigenvalues of $\varphi$ we will mean the collection of all eigenvalues of all $\varphi_i$.

There is another way of determining these eigenvalues: Consider again the torsion subgroup 
$T$ of $N$ and the induced endomorphism $\bar\varphi$ on the torsion free nilpotent group 
$M=N/T$. This latter group $M$ has a radicable hull (or rational Malcev completion) $M^\Q$. 
The endomorphism $\bar\varphi$ uniquely extends to this radicable hull $M^\Q$. We will denote this extension 
also by $\bar\varphi$.
Corresponding to this radicable hull there is a rational Lie algebra ${\mathfrak m}$. Moreover any endomorphism 
$\psi$ of $M^\Q$ induces a Lie algebra endomorphism $\psi_\ast$ of  ${\mathfrak m}$. 
The eigenvalues of $\varphi$ are then exactly the eigenvalues of the linear map $\bar\varphi_\ast$. 
In fact, if we use $\bar\varphi_{\ast,i}$ to denote the induced map on $\gamma_i({\mathfrak m})/\gamma_{i+1}(
{\mathfrak m})$, then the eigenvalues of $\varphi_i$ are exactly those of $\bar\varphi_{\ast,i}$.
 
 We remark here that one can equally well work with the real Malcev completion and its real associated Lie algebra.
 \medskip

Using this terminology of eigenvalues we have (\cite[Lemma 2.1]{DG}):
\begin{lemma}\label{eigenvaluecriterium}
Let $N$ be a finitely generated nilpotent group and $\varphi$ an endomorphism of $N$, then 
\[ R(\varphi)=\infty \Leftrightarrow 1 \mbox{ is an eigenvalue of } \varphi.\] 
\end{lemma}

In the section on non orientable surface groups we will need the following result, which generalizes Lemma 2.2 of \cite{DG}

\begin{lemma}\label{eigenvalues}
 Let $G$ be a nilpotent group which contains a free nilpotent subgroup $N_{r,c}$ as a subgroup 
of finite index. Let $\varphi$ be an endomorphism of $G$ and denote by $\varphi_i$ the induced 
endomorphisms of $\gamma_i(G)/\gamma_{i+1}(G)$. Let 
\[ \lambda_1, \lambda_2, \ldots, \lambda_r \in \C\] 
denote the eigenvalues of $\varphi_1$ $($listed as many times as their multiplicity$)$. Then any eigenvalue $\mu$ of $\varphi_i$ 
$($with $2\leq i \leq c$$)$ can be written as an $i$--fold  product 
\[ \mu =\lambda_{j_1} \lambda_{j_2} \cdots \lambda_{j_i}.\]
Moreover, any such a product in which not all of the $j_i$ are equal does appear as an eigenvalue of $\varphi_i$.
\end{lemma}
\begin{proof}
By dividing out the torsion part of $G$, we may assume that $G$ is torsion free. Now, as $G$ and 
$N_{r,c}$ are commensurable, we have that $G^\Q=N_{r,c}^\Q$ and so the corresponding Lie algebra 
is the free nilpotent Lie algebra of class $c$ on $r$ generators.   The proof then continues along the 
same lines as that of Lemma 2.2 of \cite{DG}.
\end{proof}

\begin{remark}
Note that when $\varphi$ is an automorphism, then $\lambda_1\lambda_2\ldots\lambda_r=\pm 1$.
\end{remark}

Finally, we will also need the following Lemma, which is probably well know, but we present its short proof 
here for the convenience of the reader.
\begin{lemma}\label{hopfian}
Let $N$ be a finitely generated nilpotent group and $\varphi:N\rightarrow N$ be an endomorphism.
If the induced map $N\rightarrow N/[N,N]:n\mapsto \varphi(n)[N,N]$ is surjective, then $\varphi$ is an automorphism of $N$.
\end{lemma}
\begin{proof}
The condition implies that Im$(\varphi)[N,N]=N$. As $N$ is nilpotent this implies that Im$(\varphi)=N$ and 
hence $\varphi$ is surjective (use \cite[Theorem 16.2.5]{KM}).
As a finitely generated nilpotent group (like all polycyclic-by-finite groups) is Hopfian, it follows that 
$\varphi$ has to be an automorphism of $N$. 
\end{proof}
\section{Orientable surface groups}

Consider the $2 g \times 2 g $ matrix  
\[\Omega=\left(\begin{array}{ccccccc}
0  & 1 & 0 & 0 & \cdots & 0 & 0 \\
-1 & 0 & 0 & 0 & \cdots & 0 & 0 \\
0 & 0 &  0 & 1 & \cdots & 0 & 0 \\
0 & 0 & -1 & 0 & \cdots & 0 & 0 \\
\vdots & \vdots & \vdots & \vdots  & \ddots & \vdots & \vdots \\
0 & 0 & 0 & 0 & \cdots & 0 & 1\\
0 & 0 & 0 & 0 & \cdots & -1 & 0
\end{array}\right).\] 
Note that $-\Omega = \Omega^{-1}$. 

\begin{lemma}
Let $S$ be a $2g \times 2g$--matrix with 
\[ S \Omega S^T = - \Omega \]
and let $p(x)$ denote the characteristic polynomial of $S$. Then, we have that 
\[ p(x) = \pm x^{2g} p(-\frac{1}{x} ).\]
\end{lemma}

\begin{proof}
Note that $S \Omega S^T = - \Omega $ implies that $\det(S)^2=1$ so $\det(S)=\pm 1$.
We compute
\begin{eqnarray*}
p(x) 	& = & \det (S-x  I) \\
	& = & \det(S^T - x I) \\
	& = & \det(\Omega S^{-1} \Omega + x \Omega \Omega)\\
	& = & \det(S^{-1} + x I)\\
	& = & \det(S^{-1})\det(I + x S)\\
	& = & \pm x^{2g} \det(S-(-\frac1x) I)\\
	& = & \pm x^{2g} p(- \frac1x ).
\end{eqnarray*}

\end{proof}

\begin{corollary}
If $\lambda$ is an eigenvalue of $S$, then also $\displaystyle -\frac{1}{\lambda}$ is an eigenvalue
for $S$.
\end{corollary}

\begin{remark}
The analogous result for symplectic matrices, i.e. those $S$ for which $S \Omega S^T = \Omega$, 
is well known $($and can be proved in the same way as the lemma above$)$, says that $p(x) = x^{2g} p(\frac1x)$. So the eigenvalues for a symplectic matrix come in pairs $\lambda$ and 
$\displaystyle \frac1\lambda$.
\end{remark}

Now, let us consider the fundamental group of an orientable surface of genus $g$:
\[ \pi_g=\langle a_1,b_1,a_2,b_2,\ldots,a_g,b_g\;|\;
[a_1,b_1][a_2,b_2]\cdots [a_g,b_g]=1\rangle.\]

We want to study nilpotent quotients of $\pi_g$ and the $R_\infty$ property for these quotients. Therefore, take   
an integer $c>1$ and consider any automorphism $\varphi$ 
of $\pi_g/\gamma_{c+1}(\pi_g)$. Then $\varphi$ induces an automorphism $\bar\varphi$ on $\pi_g/\gamma_2(\pi_g)\cong \Z^{2g}$.
Let $S$ be the $2g\times 2g$-matrix $S$ with integral entries which represents $\bar\varphi$ with 
respect to the free generating set (in this order) 
\[ \bar{a}_1,\bar{b}_1, \bar{a}_2, \bar{b}_2,\ldots, \bar{a}_g, \bar{b}_g\]
of $\pi_g/\gamma_2(\pi_g)$. Here we used a bar to indicate the natural projection of the generators.

\begin{lemma} With the notations introduced above, it holds that 
\[ S \Omega S^T = \pm \Omega. \]
Moreover, any matrix $S$, with integer entries satisfying 
this relation can be obtained as the matrix representing
$\bar\varphi$ for some automorphism $\varphi$ of $\pi_g/\gamma_{c+1}(\pi_g)$. $($In fact any 
such $S$ is induced by an automorphism of $\pi_g$$)$. 

\end{lemma}
\begin{proof}
This Lemma is known to hold for all automorphisms of the full group $\pi_g$ (See \cite[Corollary 5.15, page 355]{MKS})
but here we show that it also holds for all quotients $\pi_g/\gamma_{c+1}(\pi_g)$ with $c>1$. Of course,
the second part of the statement that all such $S$ can be achieved follows from the fact that all such matrices can 
be obtained as the induced map of an automorphism of $\pi_g$ (\cite[Theorem N.13, page 178]{MKS}). 

\medskip

Since any automorphism of $\pi_g/\gamma_{c+1}(\pi_g)$
induces one on $\pi_g/\gamma_3(\pi_g)$ (with the same $S$), it is enough to consider the case 
$c=2$. Let $\displaystyle N_{2g,2}=\frac{F_{2g}}{\gamma_3(F_{2g})}$ be the free 2--step nilpotent group on $2g$ generators, which we will still denote by $a_1,b_1,a_2,b_2,\ldots,a_g,b_g$. 
Let $r=[a_1,b_1][a_2,b_2]\cdots [a_g,b_g]\in N_{2c,2}$ and let $C$ be the cyclic subgroup generated by $r$, then 
$\pi_g/\gamma_3(\pi_g)=N_{2c,2}/C$. 

Any automorphism $\varphi$ of $\pi_g/\gamma_3(\pi_g)$ lifts to a morphism $\tilde{\varphi}:N_{2c,2}\rightarrow
N_{2c,2}$. This map, in turn induces a morphism $\overline{\tilde{\varphi}}:N_{2c,2}\rightarrow 
N_{2c,2}/\gamma_2(N_{2c,2})$ which is surjective.
By Lemma~\ref{hopfian} we can conclude that $\tilde{\varphi}$ is an automorphism of $N_{2c,2}$. 
As $\tilde{\varphi}$ is a lift of $\varphi$, we must have that $\tilde{\varphi}(C) \subseteq C$ and by
reasoning in the same way for $\varphi^{-1}$, we must in fact have that 
$\tilde\varphi(C)=C$  and this implies that 

\[
\tilde\varphi( r ) = r^{\pm1}
\]
or 
\[ 
\prod_{i=1}^g[\tilde\varphi(a_i),\tilde\varphi(b_i)]=
(\prod_{i=1}^g[a_i,b_i])^{\pm 1}.
\]
Since the group $N_{2g,2}$ is 2--step nilpotent, we have that (with $s_{p,q}$ the $(p,q)$-th entry of $S$): 
\[ [\tilde\varphi(a_i),\tilde\varphi(b_i)] = [\prod_{j=1}^g a_j^{s_{2j-1,2i-1}}b_j^{s_{2j,2i-1}},
\prod_{k=1}^{g} a_k^{s_{2k-1,2i}}b_k^{s_{2k,2i}}]. \]
For the convenience of writing, we will now use the additive notation in $\gamma_2(N_{2c,2})$.

The condition above then reads as 
\begin{equation}\label{bigcondition}
\sum_{i=1}^g \sum_{j=1}^g \sum_{k=1}^g 
( s_{2j-1,2i-1} s_{2k-1,2i} [a_j,a_k] + s_{2j-1,2i-1} s_{2k,2i} [a_j,b_k] +s_{2j,2i-1}s_{2k-1,2i} [b_j,a_k]+
\end{equation}
\[ \hspace*{5cm}
s_{2j,2i-1} s_{2k,2i} [b_j,b_k]) = \pm( \sum_{l=1}^g [a_l,b_l] ). \]

Now, let us look at the condition $S\Omega S^T= \pm \Omega$.
Denote by $C_1,C_2,\ldots, C_{2g}$ the columns of $S$. Then 

\[ S \Omega = (-C_2, C_1,-C_4,C_3,\ldots,-C_{2g}, C_{2g-1}). \] 
It follows that the $(p,q)$-th entry of $S\Omega S^T$ equals
\[ \sum_{i=1}^g ( - s_{p,2 i} s_{q,2i-1} + s_{p,2i-1}s_{q,2i} ).\]

Note that if $p=q$, this entry is zero, and also the $(p,p)$-th entry of $\pm \Omega$ is. So, we only have to consider the case $p\neq q$ from now onwards.\\
We now consider the different possibilities for $p$ and $q$:
\begin{itemize}
\item Both $p$ and $q$ are even. In this case the condition $S\Omega S^T= \pm \Omega$ says that the 
$(p,q)$-th entry of $S\Omega S^T$ has to be zero. Let $p=2j$ and $q=2k$, so the $(p,q)$-th entry is 
 \[ \sum_{i=1}^g ( - s_{2j ,2 i} s_{2k ,2i-1} + s_{2j ,2i-1}s_{2k ,2i} ).\]
By looking at the coefficient appearing in front of $[b_j,b_k]$ (and using the fact that $[b_k,b_j]=-[b_j,b_k])$ in 
equation \eqref{bigcondition} we find that 
\[ \sum_{i=1}^g s_{2j,2i-1}s_{2k,2i}-s_{2k,2i-1}s_{2j,2i} =0\]
which exactly says that the $(p,q)$-th entry of $S\Omega S^T$ must be zero.
\item When both $p$ and $q$ are odd, one can do the same argument by considering the coefficient appearing in front of $[a_j,a_k]$ in equation \eqref{bigcondition}. 
\item Now consider the case when $p$ is odd, say $p=2 j-1$ and $q$ is even, say $q=2 k$. By looking at the 
coefficient of $[a_j,b_k]$ in the expression \eqref{bigcondition} we find the relation
\[ \sum_{i=1}^g s_{2j-1,2i-1} s_{2k,2i} - s_{2k,2i-1}s_{2j-1,2i} = \pm \delta_{j,k}\]
(where $\delta_{j,k}$ is the Kronecker delta). \\
This is saying that  $(2j-1,2k)$-entry of $S\Omega S^T$ equals 
$\pm \delta_{j,k}$, which is what we needed to show.
\item The situation where $p$ is even and $q$ is odd reduces to the previous case.
\end{itemize}
\end{proof}

\medskip

Just as in our paper \cite{DG}, we will make use  use of the associated graded Lie ring of the group 
we are interested in, so we consider the Lie ring 
\[ L(\pi_g)= \bigoplus_{i=1}^\infty L_i(\pi_g) =  \bigoplus_{i=1}^\infty \frac{\gamma_i(\pi_g)}{\gamma_{i+1}(\pi_g)}\]
where the Lie bracket is induced by the commutators in $\pi_g$.
Then, there is a canonical homomorphism of Lie rings
\[ \psi : L(F_{2g}) \to L(\pi_g)\] 
which is induced by the projection $F_{2g}\to \pi_g$. Now, 
let $r = [a_1,b_1][a_2,b_2]\cdots [a_g,b_g] + \gamma_3(F_{2g}) \in L_2(F_{2g})$
be the element in $\gamma_2(F_{2g})/\gamma_3(F_{2g})$ corresponding to the relator 
defining $\pi_g$ and let $R$ be the ideal of  $ L(F_{2g}) $ which is generated by $r$.
Then the main result of \cite{L}
shows that $R$ is the kernel of $\psi$ and that all factor groups $\displaystyle\frac{\gamma_i(\pi_g)}{\gamma_{i+1}(\pi_g)}$ are free abelian. 

%Hence we will be able to use lemma~\ref{critnilp} again to see wether or not a given automorphism 
%of $\pi_g/\gamma_{c+1}(\pi_g)$ has infinite Reidemeister number or not. 

\medskip

Consider an automorphism $\varphi$ of $\pi_g/\gamma_{c+1}(\pi_g)$ for $c\geq 2$. There is no reason to expect that this automorphism can be lifted to an automorphism of $\pi_g$.
 However, if $S$ represents the induced automorphism on the abelianization of $\pi_g$, then we know that there does exist an automorphism $\varphi'$  of $\pi_g$, whose induced automorphism on the abelianization is also represented by $S$. Let $\varphi_1$ be the automorphism 
 of $\pi_g/\gamma_{c+1}(\pi_g)$ which is induced by $\varphi'$.
 
As the induced automorphisms on the 
quotients $\gamma_i(\pi_g)/\gamma_{i+1}(\pi_g)$ are completely determined by the induced 
automorphism on $\pi_g/\gamma_2(\pi_g)$, it follows that $\varphi$ will have infinite 
Reidemeister number if and only if $\varphi'$ has infinite Reidemeister 
number (by Lemma~\ref{eigenvaluecriterium}). So in our discussions below, we will always be allowed to assume that the automorphism 
$\varphi$ can be lifted to $\pi_g$.

Moreover, by \cite[Theorem N.10, page 176]{MKS} any automorphism of $\pi_g$
 is induced by an automorphism of $F_{2g}$.

\begin{theorem}\label{orient}
Let $g\geq 2$. Then, for any $c\geq 4$, we have that the $c$-step nilpotent quotient 
$\pi_g / \gamma_{c+1} (\pi_g)$ has property $R_\infty$. On the other hand, for $c=1,2$ and $3$,
this quotient does not have property $R_\infty$.
I.e.\ the $R_\infty$-nilpotency degree of $\pi_g$ is 4.
\end{theorem}

\begin{proof}
For $c=1$, we have that $\pi_g/\gamma_2(\pi_g)\cong \Z^{2g}$ and this group does not have property $R_\infty$.

\medskip

From now onwards, we consider the situation with $c>1$.

Let  $S_g$ be the following $2g \times 2g$--matrix:
\[ S_g=
\left(\begin{array}{ccccccc}
1  & 2 & 0 & 0 & \cdots & 0 & 0 \\
1 & 1 & 0 & 0 & \cdots & 0 & 0 \\
0 & 0 &  1 & 2 & \cdots & 0 & 0 \\
0 & 0 &  1 & 1 & \cdots & 0 & 0 \\
\vdots & \vdots & \vdots & \vdots  & \ddots & \vdots & \vdots \\
0 & 0 & 0 & 0 & \cdots & 1 & 2\\
0 & 0 & 0 & 0 & \cdots & 1 & 1
\end{array}\right).
\] 

Then $S_g$ satisfies $S_g \Omega S_g^T = - \Omega$ and $S_g$ only has eigenvalues
$\lambda= 1+\sqrt{2}$ and $-\frac{1}{\lambda}= 1 -\sqrt{2}$ (each one with multiplicity $g$).

Consider an automorphism $\varphi$ of $\pi_g$ for which the induced morphism $\bar\varphi$ on the abelianization of $\pi_g$ is given by $S_g$. As already mentioned above,
this automorphism  is induced by 
an automorphism $\tilde{\varphi}$ of $F_{2 g}$.

By Lemma~\ref{eigenvalues}, we know that the eigenvalues of induced automorphisms 
on $\gamma_2(F_{2g})/\gamma_3(F_{2g})$ resp.\  $\gamma_3(F_{2g})/\gamma_4(F_{2g})$
consist of products of 2 resp. 3 eigenvalues of $S$. It is obvious that none of these eigenvalues can be equal to 1. As the induced maps on $\gamma_2(\pi_g)/\gamma_3(\pi_g)$ resp.\  $\gamma_3(\pi_g)/\gamma_4(\pi_g)$ are themselves induced from the maps in the free case, it also follows that in this case none of the eigenvalues equals 1. Hence, when $c\leq 3$, $\varphi$ 
induces an automorphism of $\pi_g/\gamma_{c+1}(\pi_g)$ with finite Reidemeister number and so
these groups do not have property $R_\infty$.

\medskip

Now we will prove that $\pi_g/\gamma_5(\pi_g)$ does have property $R_\infty$. From this 
it then also follows that $\pi_g/\gamma_{c+1}(\pi_g)$ has property $R_\infty$ for all $c\geq 4$.

\medskip

Let $\varphi$ be any automorphism of $\pi_g/\gamma_5(\pi_g)$. From the discussion before this theorem, we can without loss of generality assume that $\varphi$ is induced by an automorphism
$\tilde{\varphi}$ of $F_{2g}$. Let $\tilde{\varphi}_\ast$ denote the graded Lie ring automorphism of $L(F_{2g})$ induced by $\tilde\varphi$ and $\varphi_\ast$ the automorphism 
of $L(\pi_g)$ induced by $\varphi$. Since $\varphi$ is induced by $\tilde\varphi$, we have 
that  $\varphi_\ast$ is induced by $\tilde\varphi_\ast$ and 
\[ \tilde\varphi_\ast(r) = \pm r,\mbox{ where }
r= [a_1,b_1][a_2,b_2]\cdots [a_g,b_g] + \gamma_3(F_{2g}) \in L_2(F_{2g})
\mbox{ as above.}\]

We have to prove that the map induced by $\varphi_\ast$ on $L(\pi_g)/\gamma_5(L(\pi_g))$
has at least one eigenvalue equal to 1. 

Let $S$ still denote the matrix representing the induced automorphism on $\pi_g/[\pi_g,\pi_g]$.

\medskip

We now distinguish two cases.

\medskip

{\bf Case 1:} $S^T \Omega S = \Omega$. This case occurs exactly when $\tilde\varphi_\ast(r)=r$.
In this situation, $S$ has eigenvalues 
\[\lambda_1,\; \frac{1}{\lambda_1},\; \lambda_2,\;\frac{1}{\lambda_2}, \ldots, \lambda_g ,
\;\frac{1}{\lambda_g}.\]
So, the induced map on $L_2 (F_{2g})= \gamma_2(F_{2g})/\gamma_3(F_{2g})$ 
has eigenvalue 1 with 
at least multiplicity $g$, since the eigenvalues of this map are products of 2 eigenvalues 
of $S$ and we have that $1=\lambda_i (1/\lambda_i)$ for all $i$. 

The projection of $L_2(F_{2g})$ onto $\displaystyle L_2(\pi_g)= \frac{L_2(F_{2g})}{L_2(F_{2g})\cap R}=
L_2(F_{2g})/\langle r \rangle$ is given by dividing out the 1-dimensional subspace 
generated by $r$. So the induced map on $L_2(\pi_g)$ has eigenvalue 1 with multiplicity at least $g-1>0$, which shows that $\varphi$ has an infinite Reidemeister number in this case.

\bigskip

{\bf Case 2:} $S^T \Omega S = -\Omega$. This case occurs exactly when $\tilde\varphi_\ast(r)=-r$.

First of all, suppose that one of the eigenvalues of $S$ is $\pm i$. In that case both $i$ and $-i$ 
appear as an eigenvalue of $S$ and the their product $i \times -i=1 $ is an eigenvalue 
for the induced map on $L_2 (F_{2g})$. Since $\tilde\varphi_\ast(r)=-r$ corresponds to  
eigenvalue $-1$, we have that the induced map on $\displaystyle L_2(\pi_g)= 
L_2(F_{2g})/\langle r \rangle$ also has an eigenvalue $1$ and so we are done in this case.

Hence we can assume that $\pm i$ do not appear as eigenvalues of $S$ and 
in this situation, $S$ has eigenvalues 
\[\lambda_1,\; \frac{-1}{\lambda_1},\; \lambda_2,\;\frac{-1}{\lambda_2}, \ldots, \lambda_g ,
\;\frac{-1}{\lambda_g}.\]
In which not all the $\lambda_i$ need to be different. 

\medskip

We will show that the induced map on $L(\pi_g)/\gamma_5(L(\pi_g))$ has 1 as an eigenvalue. Note that now $\gamma_5(L(\pi_g))$ is used to denote the 5-th term in the lower central series of the Lie ring $L(\pi_g)$. As $L(\pi_g)$ is graded and is generated by its terms in degree one, we have that $\gamma_5(L(\pi_g))=  L_5(\pi_g)\oplus L_6(\pi_g)\oplus L_7( \pi_g) \oplus \cdots$. In fact, we will show that the induced map on 
\[ \frac{L(\pi_g)}{\gamma_5(L(\pi_g))+L^{(2)}(\pi_g)}\]
has 1 as an eigenvalue. Here $L^{(2)}(\pi_g)$ denotes the second derived Lie ring, 
so we consider a metabelian and 4-step nilpotent quotient of $L(\pi_g)$.
Note that 
\[\frac{L(\pi_g)}{\gamma_5(L(\pi_g))+L^{(2)}(\pi_g)}=
\frac{L(F_{2g})}{\gamma_5(L(F_{2g}))+L^{(2)}(F_{2g})+R}=\frac{L(F_{2g})}{\gamma_5(L(F_{2g}))+L^{(2)}(F_{2g})}\raisebox{-1mm}{\Huge$/$}\bar{R}\]
where $\bar{R}$ is the natural image of the ideal $R$ in this free metabelian and 4-step nilpotent 
Lie ring $\displaystyle  
\frac{L(F_{2g})}{\gamma_5(L(F_{2g}))+L^{(2)}(F_{2g})}$. To avoid complicated notations, let us denote this last Lie ring by $L$.
If we denote by $\bar{r}$ the image of $r$ in $L$, then $\bar{R}$ is the ideal of $L$ generated 
by $\bar{r}$. 

\medskip

Now, we again consider the complexification of our Lie algebra $L^\C= L\otimes \C$. So 
$L^\C$ is the free metabelian and 4-step nilpotent 
Lie algebra over $\C$. Then 
$\bar{R}^\C=\bar{R}\otimes \C$ is the ideal of $L^\C$ generated by $\bar{r}$. Note that 
$L^\C=L^\C_1\oplus L^\C_2\oplus L^\C_3\oplus L^\C_4$ is graded and also 
$\bar{R}^\C=\bar{R}^\C_2 \oplus \bar{R}^\C_3 \oplus \bar{R}^\C_4$ is graded, since it is 
generated by an element of degree 2. 

\medskip

As in the prove of Lemma~\ref{eigenvalues} we can assume that $S$ is semi-simple.

Now, we can choose basis vectors $Y_1,Y_2,\ldots,Y_{2g}$ in $L^\C_1$ which are eigenvectors
for the map induced by $\tilde\varphi_\ast$ (i.e.\ for $S$). 
So $Y_1$ is eigenvector corresponding to $\lambda_1$, $Y_2$ to $\frac{-1}{\lambda_1}$,
$Y_3$ to $\lambda_2$, \ldots

Using these basis vectors, we can now describe $\bar{R}^\C$. 

Indeed $\bar{R}^\C_3$ is spanned by the vectors
\[ [\bar r,Y_1],\;[\bar r,Y_2], \;[\bar r,Y_3],\;[\bar r,Y_4],\ldots, [\bar r,Y_{2g-1}],\;[\bar r,Y_{2g}] \]
Note that these vectors are eigenvectors corresponding to the eigenvalues
\[-\lambda_1,\; \frac{1}{\lambda_1},\; -\lambda_2,\;\frac{1}{\lambda_2}, \ldots, -\lambda_g ,
\;\frac{1}{\lambda_g}\]
in that order.

%(AND THESE ARE LINEARLY INDEPENDENT, BUT WE DO NOT NEED THIS)

To find $\bar{R}^\C_4$ we have to take the Lie bracket of any spanning vector of $\bar{R}^\C_3$
with any generator $Y_j$. However, as we are  working in a metabelian Lie algebra, we have that 
\[ \forall i,j:\; [[\bar r,Y_i],Y_j]=[[\bar r,Y_j],Y_i] \]
and hence we find that $\bar{R}^\C_4$ is spanned by all vectors of the form 
\[ [[\bar r,Y_i],Y_j] \mbox{ where }i\leq j.\]

%(AND THESE WILL BE LINEARLY INDEPENDENT TOO BUT THIS IS HARD TO SHOW I THINK.)

Again, all of these vectors are eigenvectors. If $\lambda$ is the eigenvalue for $Y_i$ and 
$\mu$ is the eigenvalue for $Y_j$ then $- \lambda \mu$ is the eigenvalue for $[[r,Y_i],Y_j]$.

In order to finish the proof, we need to sow that $L^\C_4$ contains eigenvectors with respect 
to eigenvalue 1 which are not contained in $\bar{R}^\C_4$. Recall (\cite{C}) that $L^\C_4$ has a basis 
of eigenvectors which is given by 
\[ [[[Y_{i_1},Y_{i_2}],Y_{i_3}],Y_{i_4}]\mbox{ for } i_1> i_2\leq i_3 \leq i_4.\]
Now, if $[[\bar r,Y_i],Y_j] $ is an eigenvector for eigenvalue 1, then with the notations above, we have that $\lambda \mu=-1$. As we have already excluded the case $\lambda=\pm i$, this is only possible 
when $i\neq j$ (so $i<j$).  
To each such a vector $[[\bar r,Y_i],Y_j]$ of eigenvalue 1, we associate  the vector  
\[ [[[Y_{j},Y_{i}],Y_{i}],Y_{j}] \in L^\C_4\]
which is also an eigenvector for the eigenvalue 1. These associated vectors span already a subspace of $L^\C_4$ which is at least  
as big as the eigenspace corresponding to eigenvalue 1 
inside $\bar{R}^\C_4$. 
Since 
\[ [[[Y_2,Y_1],Y_3],Y_4] \]
is also an eigenvector of eigenvalue 1, which is linearly independent of the previous ones, we see that $L^\C_4$ does contain eigenvectors with respect to eigenvalue 1 outside $\bar{R}^\C_4$.
And this finishes the proof.
\end{proof}

%Let $G=\pi_g/\pi_g^{(2)}$, i.e. the quotient of $\pi_g$ by the $2-$th term of the  derived series of $\pi_g$. 

\begin{corollary}\label{degsolv} Let  $g\geq 2 $ and  $G=\pi_g/\pi_g^{(2)}$. The group $G/\gamma_{c+1}(G)$ has the $R_{\infty}$ property, if and only if $c\geq 4$. Further,
the $R_{\infty}$-solvability degree of $\pi_g$ is 2, i.e.  
$\pi_g/\pi_g^{(c)}$ has the $R_{\infty}$ property, if and only if $c\geq 2$.
\end{corollary}

\begin{proof} %Any automorphism of $\pi_g/\gamma_{c+1}(\pi_g)$ induces an automorphism of $\pi_g/(\gamma_{c+1}(\pi_g)\pi_g^{(2)})$. 
Note that $G/\gamma_{c+1}(G)= \pi_g/(\gamma_{c+1}(\pi_g)\pi_g^{(2)})$.\\
 For the first part, if $c=1, 2, 3$ then $\pi_g^{(2)}\subseteq \gamma_{c+1}(\pi_g)$ and so $G/\gamma_{c+1}(G)= \pi_g/(\gamma_{c+1}(\pi_g)\pi_g^{(2)}) = \pi_g/\gamma_{c+1}(\pi_g)$. The result in this case then 
 follows directly from the theorem above. 
 
 For $c>3$ the proof of the theorem above is in fact a proof showing that the group  $G/\gamma_{c+1}(G)$ has property $R_\infty$.
  If follows that $G=\pi_g/\pi_g^{(2)}$ itself also has property $R_\infty$, which already proves one step of the ``further''-part.
 
 Clearly $G/G^{(1)}$, the abelianization of $G$ which is the same of the abelianization of $\pi_g$, does not have property $R_{\infty}$. 
\end{proof}

\section{Non-orientable surface groups}

First we need a lemma concerning powers in nilpotent groups, which can be found in \cite{Se} Proposition 2, page 113. 

\begin{lemma}\label{segal}
Let $G$ be a nilpotent group of class $\leq c$. For any $s\in \N$ we have that every element of 
\[ G^{s^c}=\langle g^{s^c}\:|\; g\in G\rangle\]
is the $s$-th power of an element of $G$.  
\end{lemma}

Also  we will need  the following two Lemmas:
%Another result we will need is the following:
\begin{lemma}\label{finite index}
Let $N$ be a finitely generated nilpotent group and assume  $H$ is a subgroup of $N$, such that 
$H[N,N]$ has finite index in $N$, then $H$ has finite index in $N$.
\end{lemma}
\begin{proof}
We proceed by induction on the nilpotency class $c$ of $N$. For abelian groups, the lemma is trivially satisfied.
Now assume that $N$ is of class $c>1$ and the lemma is true for smaller nilpotency classes.
The assumption of our lemma, implies that 
\[ \frac{H [N,N]}{\gamma_c(N)}\mbox{ is of finite index in }\frac{N}{\gamma_c(N)}.\]
This can also be read as 
\[ \frac{H \gamma_c(N)}{\gamma_c(N)} \left[ \frac{N}{\gamma_c(N)},\frac{N}{\gamma_c(N)} \right]
\mbox{ is of finite index in }\frac{N}{\gamma_c(N)}.\]
By induction, we then obtain that 
\begin{equation}\label{by induction}
H\gamma_c(N) \mbox{ is of finite index in } N.
\end{equation}
In order to prove the lemma, it now suffices to show that $H\cap \gamma_c(N)$ has finite index in 
$\gamma_c(N)$. As $N$ is finitely generated, the group $\gamma_c(N)$ is also finitely generated (and abelian).
Suppose now on the contrary that the index of  $H\cap \gamma_c(N)$ in $\gamma_c(N)$ is infinite,
then there exists an element $\gamma$ of $\gamma_c(N)$ for which $\gamma^k\not \in H$ for all $k>0$.

Write $\gamma=[n_1,m_1] [n_2,m_2]\cdots [n_l,m_l]$ as a product of commutators, where each $n_i \in N$ and 
each $m_i\in \gamma_{c-1}(N)$. Note that for all integers $p,q$ we have that 
\[ [n_i,m_i]^{pq}=[n_i^p,m_i^q] . \]
Condition \eqref{by induction} implies that there is a positive integer $p$ such that for all $n\in N$ we have that 
$n^p\in H \gamma_c(N)$. So we can write $n_i^p=h_i z_i$ and $m_i^p = h'_i z'_i$ with  $h_i,h'_i\in H$ and 
$z_i,z'_i\in \gamma_c(N)$. But then 
\[ \gamma^{p^2}= [n_1^p,m_1^p] [n_2^p,m_2^p]\cdots [n_l^p,m_l^p]=
[h_1z_1,h_1'z_1'] \cdots [h_lz_l,h_l'z_l']  = [h_1,h_1'] \cdots [h_l,h_l']\in H  \]
which is a contradiction with the fact that we assumed that $\gamma^k \not \in H$ for all $k>0$.

\end{proof}

\begin{lemma}\label{free}   Let  $N_{n,c}$ be the  free $c$-step nilpotent group on $n$ generators.   The subgroup generated by a  subset $\{ x_1,...,x_n\}\subset N_{n,c}$ is 
also the free  $c$-step nilpotent group on $n$ generators if and only if the images  of the elements $x_i$ in the abelianization of $N_{n,c}$ generate a free abelian group of rank $n$.
%same ran.       morphism factors through the $c$-step nilpotent quotient of $\Gamma_{g+1}$ which 
%leads to a morphism $\psi:\Gamma_{g+1}/\gamma_{c+1}(\Gamma_{g+1})\rightarrow N_{g,c}$ with 
\end{lemma}

\begin{proof}  Only if is clear. So we let $G$ denote the subgroup generated by the subset   $\{ x_1,...,x_n\}\subset N_{n,c}$ and we assume that 
the images of  the elements $x_i$ in $N_{n,c}/\gamma_2(N_{n,c})$  generate a free abelian group of rank $n$. This implies that the image  of $G$ in the abelianization of $N_{n,c}$
has finite index in this abelianization.   By 
Lemma \ref{finite index} it  also follows that    $G$ has finite index in $N_{n,c}$.    Therefore $G$ and $N_{n,c}$ have the same Hirsch length.
Now, let $y_1,y_2, \ldots, y_n$ denote a set of free generators of $N_{n,c}$ and consider the morphism $\varphi$ mapping $y_i$ to $x_i$, which exists beacuse $N_{n,c}$ is free nilpotent of class $c$ and $G$ is also nilpotent 
of class $\leq c$.
 The image of $\varphi$ is $G$ and the kernel of 
$\varphi$ must have Hirsch length 0 (since the Hirsch length of Im$(\varphi)$ + Hirsch length of Ker$(\varphi)$ has to equal the Hirsch length of $N_{n,c}$). As $N_{n,c}$ is torsion free, it follows that $\varphi$ is an isomorphism
from $N_{n,c}$ to $G$, which finishes the proof.
\end{proof}

In this section we are working with the fundamental group of a non orientable surface of genus $g$:\[ \Gamma_g=\langle a_1,\;a_2,\;\cdots,a_g\;|\; a_1^2 a_2^2 \cdots a_g^2 =1\rangle .\]

In the sequel of this section, the number $g-1$ will play a crucial role, therefore, we will
most of the time work with $\Gamma_{g+1}$.
For a given $c$ we will use $e_i$ to denote the natural projection of $a_i$ in the quotient 
$\Gamma_{g+1}/\gamma_{c+1}(\Gamma_{g+1})$.

It is not difficult to see that the abelianization of $\Gamma_{g+1}$ is isomorphic to $\Z^g\oplus \Z_2$ and so 
contains a free abelian group on $g$ generators as a subgroup of finite index. In the next theorem we show that 
something analogously holds for all nilpotent quotients. 

\begin{theorem}
The subgroup of $\Gamma_{g+1}/\gamma_{c+1}(\Gamma_{g+1})$ generated by $e_1,\;e_2,\; \ldots ,e_g$ $($the natural 
projections of $a_1$, $a_2$, \ldots, $a_g$$)$ is a free nilpotent group of class $c$ freely generated by the $g$
elements  $e_1,\;e_2,\; \ldots ,e_g$. Moreover, this subgroup is of finite index in $\Gamma_{g+1}/\gamma_{c+1}(\Gamma_{g+1}).$
\end{theorem}
\begin{proof}
Let $N_{g,c}$ be the free $c$-step nilpotent group freely generated by $b_1$, $b_2$, \ldots, $b_g$. By 
Lemma~\ref{segal}, we know that $b_1^{2^c}b_2^{2^c}\cdots b_g^{2^c}=d^2$ for some $d\in N_{g,c}$.

Now, let $\varphi:\Gamma_{g+1}\rightarrow N_{g,c}$ be the morphism determined by 
\[ \varphi(a_1) = b_1^{2^{c-1}},\; \varphi(a_2) = b_2^{2^{c-1}},\cdots, \varphi(a_g)=b_g^{2^{c-1}},\;
\varphi(a_{g+1})=d^{-1}.\]
As $N_{g,c}$ is $c$-step nilpotent this morphism factors through the $c$-step nilpotent quotient of $\Gamma_{g+1}$ which 
leads to a morphism $\psi:\Gamma_{g+1}/\gamma_{c+1}(\Gamma_{g+1})\rightarrow N_{g,c}$ with 
\[ \psi(e_1) = b_1^{2^{c-1}},\; \psi(e_2) = b_2^{2^{c-1}},\cdots, \psi(e_g)=b_g^{2^{c-1}},\;
\psi(e_{g+1})=d^{-1}.\]

Let $\mu: \langle e_1,\;e_2,\;\cdots ,\; e_g \rangle \rightarrow N_{g,c}$ be the restriction of $\psi$.
Then, the image of $\mu$ equals the subgroup $\langle b_1^{2^{c-1}},\; b_2^{2^{c-1}},\cdots, b_g^{2^{c-1}} \rangle$ 
of $N_{g,c}$ which is also free $c$-step nilpotent on $g$ generators by Lemma \ref{free}.
Therefore, there also exists a morphism $\nu: \langle b_1^{2^{c-1}},\; b_2^{2^{c-1}},\cdots, b_g^{2^{c-1}} \rangle \rightarrow  \langle e_1,\;e_2,\;\cdots ,\; e_g \rangle$
mapping $b_i^{2^{c-1}}$ to $e_i$. It is obvious that $\mu$ and $\nu$ are each others inverse and so $\mu$ is an isomorphism, showing that 
%
%\medskip
%PROBABLY MORE DETAILS NEEDED HERE?
%\medskip
%Since the group  $ \langle e_1,\;e_2,\;\cdots ,\; e_g \rangle$  has nilpotency class $c$,  and  $\mu: \langle e_1,\;e_2,\;\cdots ,\; e_g \rangle \rightarrow 
 %\langle b_1^{2^{c-1}},\; b_2^{2^{c-1}},\cdots, b_g^{2^{c-1}} \rangle$ 
 % induces on the abelianized  an isomorphism, then it  follows at once that $\mu$ is also an isomorphism. 
% be the restriction of $\psi$.
%Then, the image of $\mu$ equals the subgroup $\langle b_1^{2^{c-1}},\; b_2^{2^{c-1}},\cdots, b_g^{2^{c-1}} \rangle$ 
%of $N_{g,c}$ which is also free $c$-step nilpotent on $g$ generators by Lemma \ref{free}.
%It follows at once that $\mu$ is in fact an isomorphism from $\langle e_1,\;e_2,\;\cdots ,\; e_g \rangle$ to 
%$\langle b_1^{2^{c-1}},\; b_2^{2^{c-1}},\cdots, b_g^{2^{c-1}} \rangle$ and 
%So   
$\langle e_1,\;e_2,\;\cdots ,\; e_g \rangle$ is free $c$-step nilpotent.

\medskip

%MORE DETAILS NEEDED HERE?

\medskip

Since $\langle e_1,\;e_2,\;\cdots ,\; e_g \rangle$ is a subgroup of $\Gamma_{g+1}/\gamma_{c+1}(\Gamma_{g+1})$
satisfying the conditions of Lemma~\ref{finite index}, we find that  
$\langle e_1,\;e_2,\;\cdots ,\; e_g \rangle$ is of finite index in $\Gamma_{g+1}/\gamma_{c+1}(\Gamma_{g+1})$.

\end{proof}
\begin{corollary}\label{hasR}
If $c\geq 2g$, then $\Gamma_{g+1}/\gamma_{c+1}(\Gamma_{g+1})$ has property $R_\infty$.
\end{corollary}
\begin{proof}
Consider any automorphism $\varphi$ of $G$ and let $\lambda_1,\lambda_2,\ldots,\lambda_g$ denote the eigenvalues of $\varphi_1$. By Lemma~\ref{eigenvalues} we know that 
$(\lambda_1\lambda_2\cdots \lambda_g)^2=(\pm 1)^2 =1 $ is an eigenvalue of $\varphi_{2g}$ and hence, 
by Lemma~\ref{eigenvaluecriterium} we can conclude that $R(\varphi)=\infty$.
\end{proof}

The following lemma is a technical one which we will need for the construction of certain automorphisms of  $\Gamma_{g+1}$.

\begin{lemma} \label{fnc} There exists a function $f : \N\times \N \rightarrow \N$
such that for any nilpotent group $G$ of class $\leq c$ and any pair of elements 
$x,y\in G$ we have that 
\[ \exists z \in \langle y, \; \gamma_2(G)\rangle \mbox{ with }
x^n y^{f(n,c)} = (xz)^n. \] 
\end{lemma}

\begin{proof} Fix $c,n\in \N$ and 
let $N_{2,c}$ be the free nilpotent group of class $c$ on two generators, say $a$ and $b$.
It is enough to show that there exists a value $f(n,c)$ and an element $d\in \langle b , \gamma_2(N_{2,c})\rangle$ such that 
\[ a^n b^{f(n,c)} = (a d)^n.\]
The case for a general group then follows by considering the homomorphism 
$\psi:N_{2,c}\rightarrow G$ determined by $\psi(a)=x$ and $\psi(b)=y$ and by taking $z=\psi(d)$.

\medskip

Let 
$a_1=a,\;a_2=b,\; a_3=[a,b],\;a_4, \; \ldots, \;a_k$
be a Malcev basis for $N_{2,c}$.

It is known that there are polynomials
% $p_i(x_1,x_2,\ldots,x_{i-1},y_1,y_2,\ldots,y_{i-1}) $ and 
$q_i(x_1,\ldots,x_{i-1},m)$ with coefficients in $\Q$ 
such that for any $x_1,x_2,\ldots, x_k,
%y_1,y_2,\ldots,y_n,
m\in \Z$ we
have that 

%\begin{eqnarray*}
% \lefteqn{a_1^{x_1}a_2^{x_2} \ldots a_n^{x_n} a_1^{y_1}  a_2^{y_2} \ldots a_n^{y_n}}\\
% & = & a_1^{x_1+y_1} a_2^{x_2+y_2+p_2(x_1,y_1)} \ldots 
% a_i^{x_i+y_i+ p_i(x_1,\ldots,y_{i-1})} \ldots a_n^{x_n+y_n+p_n(x_1,\ldots,y_{n-1})}\\
% \lefteqn{and}\\
\[
(a_1^{x_1}a_2^{x_2} \ldots a_k^{x_k})^m
 =  
a_1^{m x_1} a_2^{m x_2 + q_2(x_1,m)}\ldots  
a_i^{m x_i + q_i (x_1,\ldots,x_{i-1}, m)} \ldots a_k^{m x_k + q_n(x_1,x_2,\ldots,x_{k-1},m)}.
\]
%\end{eqnarray*}
(In fact  %$p_2$ and 
$q_2=0$). Now, for any $m$ we consider the equation in the unknowns $z_2,z_3,z_4,\ldots,z_k$:
\[ a^n b^m = (a_1 a_2^{z_2} a_3^{z_3}\ldots a_k^{z_k})^n.\]
This equation is equivalent to a set of equations:
\[ \left\{
\begin{array}{l}
z_2= \frac{1}{n} m\\
0=z_3 +  \frac{1}{n} q_3(1,z_2,n)\\
0=z_4 + \frac{1}{n}q_4(1,z_2,z_3,n)\\
\;\;\;\;\vdots\\
0=z_k+\frac{1}{n}q_k(1,z_2,\ldots, z_{k-1}).
\end{array}\right.\]
Over $\Q$ this system of equations can be solved in a recursive way and we see that there are 
polynomials $p_i(x)$ in one variable such that the solution of the above equation is given 
by $z_i = p_i(m)$. Now these polynomials do not have a constant term (because if $m=0$, then 
all $z_i=0$). Hence, if we take $m$ equal to the least common multiple of all denominators
 appearing in the polynomials $p_i$, then all the $z_i=p_i(m)$ will lie in $\Z$. So by taking $f(n,c)=m$ for this $m$, the lemma is proved, since in this case we have that  
\[ d=a_2^{z_2} a_3^{z_3}\ldots a_k^{z_k}\in \langle b, \gamma_2(N_{2,c})\rangle.\]
\end{proof}

\begin{theorem}\label{map varphi}
Let $c>0$ and $g>1$ be fixed.
For any  integer $k\in \Z$, there exists an automorphism
$\varphi_k$ of  $\Gamma_{g+1}/\gamma_{c+1}(\Gamma_{g+1})$ with 
\[ \varphi_k(e_1) = e_2,\; 
\varphi_k(e_2) =e_3,\; \varphi_k(e_3) = e_4 ,\;\ldots, \varphi_k(e_{g-1})=e_{g} \mbox{ and }
\varphi_k(e_{g})=e_1 e_{g}^{m},\]
with $m=(k f(2,c))^c$. Here $f(2,c)$ is the function defined in Lemma~\ref{fnc}.
\end{theorem}

\begin{proof}
Note that 
\[ e_1^{-2}(e_1 e_g^m)^2 = (e_1^{-1} e_g^m e_1) e_g^m\]
which is a product of $m$-th powers in  $\langle e_1,\ldots, e_g\rangle$. Since $m=(k f(2,c))^c$, 
Lemma~\ref{segal} tells us that there is an element $\omega\in \langle e_1,\ldots, e_g\rangle$ such that 
\[  e_1^{-2}(e_1 e_g^m)^2 = \omega^{k f(2,c) }.\]
By Lemma~\ref{fnc}, we know that there is an element 
$\beta \in \langle e_1,\ldots, e_g, \gamma_2(G)\rangle$ with 
\[ (e_1^{-2} e_{g+1}^{-1} e_1^2)^2 \omega^{k f(2,c)} = ( e_1^{-2} e_{g+1}^{-1} e_1^2 \beta )^2.\]

\medskip

Let $\alpha=e_1^{-2} e_{g+1}^{-1} e_1^2 \beta$. 

Consider $F_{g+1}$ the free group on $g+1$ generators,
say $a_1,a_2,\ldots, a_{g+1}$, and define a morphism 
$\psi_k:F_{g+1} \to \Gamma_{g+1}/\gamma_{c+1}(\Gamma_{g+1})$ 
which is determined by 
\[ \psi_k(a_1)= e_2,\; \psi_k(a_2) =e_3,\;\ldots,\; \psi_k(a_{g-1})=e_g,\;
\;\psi_k (a_g)=e_1 e_g^m,\;
\psi_k(a_{g+1}) =\alpha^{-1}.\]

We now compute that 
\begin{eqnarray*}
\psi_k(a_1^2a_2^2a_3^2\cdots a_{g+1}^2) 
& = &   e_2^2e_3^2\ldots e_g^2 (e_1 e_g^{m})^2 \alpha^{-2} \\
& = & e_1^{-2} ( e_1^2e_2^2\ldots e_g^2) e_1^{2} \left( e_1^{-2} (e_1 e_g^{m})^2\right) \alpha^{-2}\\
& = & (e_1^{-2} e_{g+1}^{-1} e_1^2)^2 \omega^{k f(2,c)} \alpha^{-2}\\
& = & \left( e_1^{-2} e_{g+1}^{-1} e_1^2 \beta \right)^2 \alpha^{-2}\\
& = & 1.
\end{eqnarray*}
Since $\psi_k$ maps the defining relator for $\Gamma_{g+1}$ to 1, $\psi_k$ determines a morphism from $\Gamma_{g+1} \to \Gamma_{g+1}/\gamma_{c+1}(\Gamma_{g+1})$ and 
hence also a morphism 
$\varphi_k: \Gamma_{g+1}/\gamma_{c+1}(\Gamma_{g+1}) \to 
\Gamma_{g+1}/\gamma_{c+1}(\Gamma_{g+1})$.

To see that $\varphi_k$ is an automorphism, it is enough, by Lemma~\ref{hopfian}, that the induced map to the abelianization is surjective. But this is now trivial to see, by the form 
of the images of the generators. 

\end{proof}

\begin{theorem}\label{existmap}
For any $c>0$ and any $g>1$ there exists an automorphism $\varphi$ of 
$\Gamma_{g+1}/\gamma_{c+1}(\Gamma_{g+1})$ such that the induced automorphism 
$\varphi_1$ of the abelianization of $\Gamma_{g+1}$ has determinant $-1$, has one real eigenvalue of 
modulus $>1$ $($and multiplicity 1$)$ and all other eigenvalues are of modulus $<1$.
\end{theorem}

Before proving this theorem, we need some more technical results.

The first lemma is a trivial observation.
\begin{lemma}\label{map lambda} 
For any $g>1$, the group $\Gamma_g$ admits the automorphism 
$\lambda$ which is determined by 
$\lambda(a_1) = a_2$, $\lambda(a_2)= a_3$, \ldots , $\lambda(a_{g-1})=a_g$ and
$\lambda(a_g)=a_1$. 

\end{lemma}

The next lemma is  a technical one.

\begin{lemma}\label{matrices}
Let $A$ and $L$ be $g\times g$ matrices which are given by 
\[ A=\left(\begin{array}{cccccc}
0 & 0 & \cdots & 0 & 0 & 1 \\
1 & 0 & \cdots & 0 & 0 & 0 \\
0 & 1 & \cdots & 0 & 0 & 0 \\
\vdots & \vdots & \ddots & \vdots &\vdots & \vdots \\
0 & 0 & \cdots & 1 & 0 & 0 \\
0 & 0 & \cdots & 0 & 1 & m
\end{array}\right) \mbox{ and }
L=\left(\begin{array}{cccccc}
0 & 0 & \cdots & 0 & 0 & -1 \\
1 & 0 & \cdots & 0 & 0 & -1 \\
0 & 1 & \cdots & 0 & 0 & -1 \\
\vdots & \vdots & \ddots & \vdots &\vdots & \vdots \\
0 & 0 & \cdots & 1 & 0 & -1\\
0 & 0 & \cdots & 0 & 1 & -1
\end{array}\right). \]
Then for $m$ big enough, the matrix $LA^{g-1}$ has one real eigenvalue of modulus $>1$ $($and multiplicity 1$)$, while all other eigenvalues are of modulus $<1$. 
\end{lemma}

\begin{proof}
To multiply a matrix on the left with $A$ is to put the bottom row on top, move rows 1 to $g-2$ one row down and replace the last row by row $g-1$  plus $m$ times the last row. Hence, we easily see 
that
\[ A^2 =  \left(\begin{array}{cccccc}
0 & 0 & \cdots & 0 & 1 & m \\
0 & 0 & \cdots & 0 & 0 & 1 \\
1 & 0 & \cdots & 0 & 0 & 0 \\
0 & 1 & \cdots & 0 & 0 & 0 \\
\vdots & \vdots & \ddots & \vdots &\vdots & \vdots \\
0 & 0 & \cdots & 1 & m & m^2 
\end{array}\right),\; 
A^3 =  \left(\begin{array}{cccccc}
0 & 0 & \cdots & 1 & m & m^2 \\
0 & 0 & \cdots & 0 & 1 & m \\
0 & 0 & \cdots & 0 & 0 & 1 \\
\vdots & \vdots & \ddots & \vdots &\vdots & \vdots \\
0 & 0 & \cdots & 0 & 0 & 0 \\
0 & 0 & \cdots & m & m^2 & m^3 
\end{array}\right),\;\ldots\]
and finally
\[ A^{g-1}=\left(\begin{array}{cccccccc}
0 &1 & m & m^2 & \cdots & m^{g-4} & m^{g-3} & m^{g-2}\\
0 & 0 & 1 & m & \cdots & m^{g-5} & m^{g-4} & m^{g-3}\\
0 & 0 & 0 & 1 & \cdots & m^{g-6} & m^{g-5} & m^{g-4}\\
\vdots & \vdots & \vdots & \vdots  & \ddots & \vdots & \vdots & \vdots \\
0 & 0 & 0 & 0 & \cdots & 1 & m & m^2 \\ 
0 & 0 & 0 & 0 & \cdots & 0 & 1 & m \\
0 & 0 & 0 & 0 & \cdots & 0 & 0 & 1\\
1 & m & m^2 & m^3 & \cdots & m^{g-3} & m^{g-2} & m^{g-1}
\end{array}\right) .\] 
So 
\[ LA^{g-1}=\left(\begin{array}{ccccccc}
-1 & -m & -m^2 & -m^3 & \cdots  & -m^{g-2} & -m^{g-1}\\
-1 & 1-m & m-m^2 & m^2-m^3 & \cdots & m^{g-3}- m^{g-2} & m^{g-2} - m^{g-1}\\
-1 & -m & 1 - m^2 & m - m^3 & \cdots & m^{g-4} - m^{g-2} & m^{g-3} - m^{g-1}\\
-1 & -m & -m^2 & 1 - m^3 & \cdots & m^{g-5} -m^{g-2} & m^{g-4}-m^{g-1}\\
\vdots & \vdots & \vdots & \vdots & \ddots & \vdots & \vdots\\
-1 & -m & -m^2 & -m^3 & \cdots & m- m^{g-2} & m^2 -m^{g-1}\\
-1 & -m & -m^2 & -m^3 & \cdots & 1-m^{g-2} & m - m^{g-1}\\
-1 & -m & -m^2 & -m^3 & \cdots & -m^{g-2} & 1 - m^{g-1}\\
\end{array}\right) .\] 
The eigenvalues of this matrix are the roots of the characteristic polynomial
\[ p(x)= \det( x I_g - LA^{g-1}).\]
To be able to compute this characteristic polynomial, we consider the $g\times g$-- matrix 
\[ K=\left( \begin{array}{ccccccc}
1 & -m & 0 & 0  & \cdots & 0 & 0 \\
0 & 1 & -m & 0 & \cdots & 0 & 0 \\
0 & 0 & 1 & -m & \cdots & 0 & 0 \\
0 & 0 & 0 & 1 & \cdots & 0 & 0 \\
\vdots & \vdots & \vdots & \vdots & \ddots & \vdots & \vdots \\
0 & 0 & 0& 0 & \cdots & 1 & -m\\
0 & 0 & 0 & 0 & \cdots & 0 & 1
\end{array}\right)\]
of determinant 1. So we also have that 
\[ p(x) = \det( (x I_g - LA^{g-1}) K) .\]
The effect of multiplying a $g\times g$--matrix $B$ with $K$ is to replace the $i$-th column of $B$ ($i\geq 2$) by 
the $i$-th column minus $m$ times the $(i-1)$--th column. As a result, we find that 
\[ p(x)=\det \left( 
\begin{array}{ccccccc}
1+x & -mx & 0 & 0 & \cdots & 0 & 0 \\
1 & -1+x & -mx & 0 & \cdots & 0 & 0 \\
1 & 0 & -1+x & -mx & \cdots & 0 & 0 \\
1 & 0 & 0 & -1 + x & \cdots & 0 & 0 \\
\vdots & \vdots & \vdots & \vdots & \ddots & \vdots & \vdots \\
1 & 0 & 0 & 0 & \cdots &  -1+x & -mx \\
1 & 0 & 0 & 0 & \cdots & 0 & -1 +x \\
\end{array}\right).\]
By developing this determinant with respect to the first column we obtain:
\[ p(x)= (1+x)(x-1)^{g-1} + \sum_{k=1}^{g-1} (mx)^k (x-1)^{g-1-k}. \]

If we denote by $a_n$, the coefficient of $p(x)$ corresponding to $x^n$, then we see that 
\begin{itemize}
\item $a_g=1$.
\item $a_{g-1}=m^{g-1} + $ some polynomial of degree $<g-1$ in $m$.
\item $a_n$ is a polynomial of degree $\leq n$ in $m$ (for all $n<g-1$). 
\end{itemize}
It follows that for $m$ large enough, we have that 
\[ |a_{g-1}| > |a_{g-2}| + \cdots + |a_2| + |a_1| + 2\]
from which we can conclude that $p(x)$ has one real root outside the unit disk and all 
other roots lie strictly inside the unit disk. (E.g.\ see Lemma 2.4 in \cite{DG}.)   
\end{proof}

We are now ready for:

\begin{proof}[Proof of Theorem~\ref{existmap}]
Let $T\cong \Z_2$ be the torsion subgroup of $\Gamma_{g+1}/\gamma_2(\Gamma_{g+1})\cong \Z^g \oplus \Z_2$ 
and let $(\Gamma_{g+1}/\gamma_2(\Gamma_{g+1}))/T$ be the torsion free abelianization of $\Gamma_{g+1}$.
It is not so difficult to see that the images of the generators $a_1$, $a_2$, \ldots, $a_g$ form a basis of this 
free abelian quotient. 

With respect to this basis, the matrix corresponding to the map $\lambda$ from Lemma~\ref{map lambda}
is exactly the matrix $L$ from Lemma~\ref{matrices}. On the other hand, the map $\varphi_k$ from 
Theorem~\ref{map varphi} has the matrix $A$ as its representation. 
Note that $\det(L)=\pm 1$ and $\det(A)=\pm 1$ have opposite signs (exact values depending on $g$ being even or odd), but we always have that $\det(LA^{g-1})=-1$. 
Hence by taking $k$ (and thus also $m$) large enough, we have that 
\[ \varphi = \lambda\circ \varphi_k^{g-1}\]
is the automorphism we are looking for.  
\end{proof}
\begin{corollary}\label{notR}
Let $g>1$ and $c < 2 g$, then $\Gamma_{g+1}/\gamma_{c+1}(\Gamma_{g+1})$ does not have property $R_\infty$.
\end{corollary}
\begin{proof}
Fix a $c<2 g$ and consider the automorphism $\varphi$ from Theorem~\ref{existmap}. 
The conditions on the eigenvalues $\lambda_1, \lambda_2, \ldots , \lambda_g$ 
of $\varphi_1$ (the induced map on the abelianization) imply that none of the products 
$\lambda_{j_1} \lambda_{j_2} \cdots \lambda_{j_i} $ with $i<2 g$ equals 1. 

By Lemmas \ref{eigenvaluecriterium} and 
\ref{eigenvalues} this implies that $R(\varphi)<\infty$.
\end{proof}

From Corollaries \ref{hasR}  and \ref{notR} we can write the main result of this section.

\begin{theorem}\label{nonorient}
Let $g\geq 2$. Then, for any $c\geq 2g$, we have that the $c$-step nilpotent quotient 
$\Gamma_{g+1} / \gamma_{c+1} (\Gamma_{g+1})$ has property $R_\infty$. On the other hand, for $c<2g$,
then $\Gamma_{g+1} / \gamma_{c+1} (\Gamma_{g+1})$  does not have property $R_\infty$.\\
I.e.\ the $R_\infty$-nilpotency degree of $\Gamma_{g+1}$ is $2g$.
\end{theorem}
%\bigskip
%
%\begin{center}
%{\bf \Large Conclusion:}\\[2mm]
%Let $g\geq 2$:\\
%$\Gamma_{g+1}/\gamma_{c+1}(\Gamma_{g+1})$ has property $R_\infty \Leftrightarrow c\geq 2 g$.
%
%\end{center}

\end{document}